\documentclass[10pt]{article}
\usepackage{amsmath,amssymb}
\usepackage[T1]{fontenc}
\usepackage[utf8]{inputenc}
\usepackage{amsthm,stmaryrd}
\usepackage[francais]{babel} 
\usepackage{fancyhdr,lastpage}
\usepackage{multirow}

\newtheorem{théorème}{Théorème}[section]
\newtheorem{lemme}[théorème]{Lemme}
\newtheorem{proposition}[théorème]{Proposition}
\newtheorem{corollaire}[théorème]{Corollaire}

\newtheorem*{théorème_vide}{Théorème}

\newtheorem*{définition_vide}{Définition}

\theoremstyle{remark}

\numberwithin{equation}{section}

\DeclareMathOperator{\ric}{Ric}
\DeclareMathOperator{\scal}{Scal}
\DeclareMathOperator{\R}{R}
\DeclareMathOperator{\Se}{S}

\DeclareMathOperator{\tr}{tr}

\title{Exemples de sous variétés de $\mathbb{S}^n$; les variétés proprement biharmoniques avec indice et les variétés proprement C--harmoniques.}
\author{Vincent Bérard}

\begin{document}

\maketitle

\section{Introduction}

Soient $(M,g)$ et $(N,h)$ deux variétés riemanniennes de dimension $m$ et $n$, on définit l'énergie des applications de $M$ dans $N$ par la fonctionnelle $ E(\varphi) = \frac{1}{2} \int_M |T\varphi|^2_{g,h} \,dvol_g$, où $T\varphi$ désigne l'application tangente de $\varphi$. Ses points critiques sont appelés les applications harmoniques de $M$ dans $N$ et un résultat classique les caractérise comme étant les solutions de l'EDP elliptique d'ordre $2$ non--linéaire $\delta T\varphi=0$, où $\delta$ désigne la divergence canonique du fibré $\Omega^1(M) \otimes \varphi^* TN$. Les applications biharmoniques sont définis comme étant les points critiques de la biénergie $E^2(\varphi) := \frac{1}{2} \int_M |\delta T\varphi|_g^2 \, dvol_g$, c'est--à--dire, les solutions de l'EDP elliptique d'ordre $4$ non--linéaire suivante\,:
\begin{equation}
 \label{equation_biharmonique}
 \Delta \delta T\varphi - \Se(\delta T\varphi) = 0,
\end{equation}
avec $\Se$ l'endomorphisme de $\varphi^* TN$ défini par $\Se(X) = \tr \R^h_{\textstyle{X,T\varphi}} T\varphi$ où la trace est prise par rapport à $g$ et $\R^h$ est le tenseur de courbure de $(N,h)$. Il s'agit d'une classe d'applications qui englobe les applications harmoniques et qui permet par exemple, de donner une nouvelle démonstration du théorème d'Eells et Sampson sur l'existence d'applications harmoniques dans les classes d'homotopie (voir \cite{MR2124627}). On sait que les applications harmoniques n'existent pas toujours (voir l'article d'Eells et Wood \cite{MR0420708} pour des exemples d'obstructions topologiques) et un des principaux objectifs de cette théorie est de prouver l'existence d'application biharmonique dans ces cas là.

\vskip 0.2cm

L'énergie est une fonctionnelle invariante conforme par rapport à $g$ quand $M$ est de dimension $2$, ce n'est pas le cas pour les autres dimensions. Néanmoins, il existe un analogue en dimension supérieure, c'est--à--dire une fonctionnelle invariante conforme pour les applications entre deux variétés riemanniennes, dont la variété de départ est de dimension $n$ paire. Ses points critiques satisfont une EDP elliptique d'ordre $n$ non--linéaire qui est invariante conforme par rapport à la variété de départ, on les appelle les applications conforme--harmoniques. On pourra consulter \cite{MR2722777} pour avoir plus de détails sur l'existence et des exemples de telles applications. Il existe une version C--harmonique en dimension $4$ du théorème d'Eells et Sampson cité ci--dessus ; récemment Biquard et Madani on montré dans \cite{2011arXiv1112.6130B}, que sous certaines hypothèses de courbures de $(M^4,g)$ et $(N,h)$, il existe une application C--harmonique dans chaque classe d'homotopie de $C^{\infty}(M,N)$. La construction de ces EDP est explicite et en dimension $4$ on obtient qu'une application $\varphi$ est C--harmonique de $(M^4,[g])$ dans $(N,h)$ si et seulement si
\begin{equation}
 \label{equation_C-harmonique_dim4_intro}
 \Delta \delta T\varphi - \Se(\delta T\varphi) + \delta (\frac{2}{3}\scal - 2\ric) T\varphi = 0,
\end{equation}
où $\scal$ et $\ric$ désignent respectivement la courbure scalaire et le tenseur de Ricci de $g$ . Remarquons que par rapport à l'équation de la simple biharmonicité (\ref{equation_biharmonique}), il y a un terme en courbure en plus dans notre équation (\ref{equation_C-harmonique_dim4_intro}) qui est d'ordre $2$ en l'application. Intéressons nous au cas particulier où la variété de départ est la sphère $\mathbb{S}^n$, ce terme supplémentaire est alors un multiple du laplacien de $\varphi$. C'est ce point de vue qu'on va généraliser, en regardant ce type d'équation où l'on se fixe arbitrairement ce facteur multiplicatif sur une variété $M$ de dimension quelconque. Il s'agit d'une généralisation des applications harmoniques qui coïncide avec les applications biharmoniques quand $k=0$.
\begin{définition_vide}
 Soit $k$ un nombre réel, une application $\varphi$ de $M$ dans $\mathbb{S}^n$ est dite biharmonique d'indice $k$ si elle est solution de l'équation suivante\,:
 \begin{equation*}
  \Delta \delta T\varphi - \Se(\delta T\varphi) + k\, \delta T\varphi = 0.
 \end{equation*}
 On appelle variété biharmonique d'indice $k$ de $\mathbb{S}^n$, une sous variété de $\mathbb{S}^n$ dont l'inclusion est biharmonique d'indice $k$.
\end{définition_vide}
Si l'inclusion est harmonique, la variété est biharmonique pour n'importe quel indice. Une question naturelle est alors de déterminer les variétés biharmoniques d'indice $k$ dont l'inclusion n'est pas harmonique, on les appellera les variétés proprement biharmoniques d'indice $k$. Il s'agit là, de la motivation principale de cet article.

\vskip 0.2cm

Considérons que $M$ est une sous variété de $N$, on note $i$ son inclusion. On identifie $TM$ à un sous fibré de $i^*TN$ et on note $NM$ son complémentaire orthogonale par rapport à $g$ qu'on appelle fibré normal (de l'inclusion $i$). On note $B$ la seconde forme fondamentale de $i$, $A$ l'opérateur de Weingarten et $H$ la courbure moyenne définie par $H := \frac{1}{m} \tr B = -\frac{1}{m} \delta Ti$, où la trace est prise par rapport à $g$. On rappelle qu'une sous variété est dite pseudo--ombilicale si $A_H = |H|^2 Id_{TM}$. Dans cet article, nous étudierons dans un premier temps les variétés proprement biharmoniques avec indice de la sphère. Un premier résultat est de remarquer que la dimension de la sous variété joue un rôle critique pour l'indice de biharmonicité. En effet, nous verrons avec la proposition \ref{proposition_variétés_biharmoniques_indicées_compactes}, que les sous variétés compactes de dimension $m$ ne sont pas proprement biharmoniques pour les indices strictement supérieurs à $m$. De plus, si $M$ est pseudo--ombilicale, alors il existe un intervalle $I = [\delta, m]$, tel que $M$ est proprement biharmonique d'indice $k \in I$ si et seulement si la courbure moyenne de $M$ est parallèle et de norme égale à une constante qui dépend de $k$ et $m$.

\vskip 0.2cm

Pour les variétés proprement biharmoniques de la sphère, il existe deux exemples fondamentaux qui sont l'hypersphère de rayon $\frac{1}{\sqrt{2}}$ qui est pseudo--ombilicale et certains tores de Clifford qui ne sont pas pseudo--ombilicals. Les sous variétés biharmoniques avec indice généralisent cette situation en faisant apparaître de manière naturelle toutes les hypersphères différentes de l'équateur et une classe plus vaste de tores de Clifford, qui sont tous proprement biharmoniques pour un certain indice. Pour les premières, remarquons qu'une sous variété d'une hypersphère différente de l'équateur ne peut être à la fois harmonique dans l'hypersphère et dans la sphère. On obtient alors le théorème \ref{théorème_hypersphère_1} ; une sous variété d'une hypersphère différente de l'équateur est harmonique dans cette hypersphère si et seulement si elle est proprement biharmonique pour un certain indice (qui dépend exclusivement de sa dimension et de l'hypersphère en question) dans la sphère. En faisant les hypothèses sur la courbure moyenne de la proposition \ref{proposition_variétés_biharmoniques_indicées_compactes} pour les variétés pseudo--ombilicales, on obtient le théorème \ref{théorème_hypersphère_2}\,:
\begin{théorème_vide}
 Soit $M$ une sous variété pseudo--ombilicale de courbure moyenne parallèle et de norme constante non nulle de $\mathbb{S}^n$, alors\,
 \begin{enumerate}
  \item $M$ est proprement biharmonique d'indice $m(1-|H|^2)$ dans $\mathbb{S}^n$,
  \item $M$ est harmonique dans $\mathbb{S}(\frac{1}{\sqrt{1+|H|^2}})$.
 \end{enumerate}
\end{théorème_vide}
L'indice de biharmonicité d'une sous variété compacte non--harmonique est majoré par sa dimension, en ce qui concerne les tores de Clifford, on a une borne optimale qui est égale à $n_1 + n_2 - 2\sqrt{n_1 n_2}$ (proposition \ref{proposition_borne_tore}). Il existe le même genre de résultats que pour les sous variétés des hypersphères, il faut alors considérer des sous variétés produits du tore de Clifford (voir théorème \ref{théorème_tore_1}).

\vskip 0.2cm

Dans un deuxième temps, on va étudier les sous variétés de la sphère, dont l'inclusion est proprement C--harmonique. On les appelle les variétés proprement C--harmoniques de la sphère qui sont donc de dimension paire et on se contentera de leur étude en basses dimensions, c'est--à--dire pour $m=4$ ou $m=6$. On obtient un résultat de rigidité pour les sous variétés compactes pseudo--ombilicales de dimension $m$ de $\mathbb{S}^n$ avec les propositions \ref{proposition_compacte_C--harmonique_4} et \ref{proposition_compacte_C--harmonique_6}; elle ne sont pas proprement C--harmoniques si $n \geq m+2$. En revanche, quand $n=m+1$ et que la courbure moyenne est minorée par un certaine constante strictement positive, alors elle est proprement C--harmonique si et seulement si sa courbure moyenne est parallèle et de norme constante non nulle. En basse dimension, on obtient l'existence et l'unicité d'hypersphères qui soient proprement C--harmonique, alors que tous les tores de Clifford ne sont pas proprement C--harmonique. Il convient de remarquer que le cas $m=4$ est bien évidemment équivalent à la biharmonicité pour un certain indice, alors que pour $m=6$, il s'agit d'une équation d'ordre $6$ qui a pour effet d'augmenter de manière significative, la complexité des calculs.

\subsection{Les variétés proprement biharmoniques avec indice dans $\mathbb{S}^n$}

\subsubsection{Définitions et propriétés}

Soient $(M,g)$ et $(N,h)$ deux variétés riemanniennes de dimension $m$ et $n$, on se donne $k$ un nombre réel, la définition suivante généralise la notion d'harmonicité et se confond avec la biharmonicité quand $k=0$\,:
\begin{définition_vide}
 \label{définition_applications_biharmoniques_indicées}
 Une application $\varphi$ de $M$ dans $N$ est dite biharmonique d'indice $k$ si elle est solution de l'équation elliptique d'ordre $4$ non--linéaire suivante\,:
 \begin{equation*}
  \Delta \delta T\varphi - \Se(\delta T\varphi) + k\, \delta T\varphi = 0,
 \end{equation*}
 avec $\Se$ l'endomorphisme de $\varphi^* TN$ défini par $\Se(X) = \tr \R^h_{\textstyle{X,T\varphi}} T\varphi$, où la trace est prise par rapport à $g$ et $\R^h$ est le tenseur de courbure de $(N,h)$. 
\end{définition_vide}

Il est facile de montrer que les applications biharmoniques d'indice $k$ sont les points critiques de la fonctionnelle suivante\,:
\begin{equation*}
 \mathcal{E}^2_k (\varphi) = \int_M |\delta T\varphi|^2 + k |T\varphi|^2 \,dvol_g.
\end{equation*}

Si $(M,g)$ est une variété d'Einstein de dimension $4$, alors les applications biharmoniques d'indice $\frac{1}{6}\scal^g$ sont les applications C--harmoniques de $(M,[g])$ dans $(N,h)$, on y reviendra dans la troisième partie. Une remarque très importante est de noter qu'une application harmonique est biharmonique pour n'importe quel indice, la bonne notion à étudier est donc la \og propre biharmonicité avec indice\fg, on a ainsi l'unicité de l'indice $k$.

\begin{définition_vide}
 Une application biharmonique d'indice $k$ de $M$ dans $N$ l'est proprement si elle n'est pas harmonique de $M$ dans $N$.
\end{définition_vide}

Supposons maintenant que $N$ est la sphère $\mathbb{S}^n$ munie de sa métrique canonique et que $(M,g)$ est une variété riemannienne quelconque, alors $\varphi$ est biharmonique d'indice $k$ de $M$ dans $\mathbb{S}^n$ si et seulement si
\begin{equation*}
 (\Delta - m + k)\, \delta T\varphi = 0.
\end{equation*}
Dans toute la suite, on va considèrer que $M$ est une sous variété de la sphère $\mathbb{S}^n$, on définit alors les variétés biharmoniques avec indice de la façon suivante\,:
\begin{définition_vide}
 Une sous variété $M$ de $\mathbb{S}^n$ est dite biharmonique d'indice $k$ dans $\mathbb{S}^n$, si son inclusion est biharmonique d'indice $k$ de $M$ dans $\mathbb{S}^n$, c'est--à--dire si et seulement si
 \begin{equation}
  \label{equation_biharmonique_indicée}
  (\Delta - m + k)\, H = 0,
 \end{equation}
 où $H$ désigne la courbure moyenne de $M$ dans $\mathbb{S}^n$. Si de plus, l'inclusion n'est pas harmonique, alors on dira que $M$ est proprement biharmonique d'indice $k$ dans $\mathbb{S}^n$.
\end{définition_vide}
 La proposition suivante caractérise les variétés biharmoniques avec indice\,:
\begin{proposition}
 \label{proposition_biharmonique_indicée}
 La variété $M$ est biharmonique d'indice $k$ dans $\mathbb{S}^n$ si et seulement si on a le système suivant\,:
\begin{center}
 $\left\{\begin{array}{ll}
  \Delta^N H - (m - k)\, H + \tr B\big(\,\cdot\,,A_H (\,\cdot\,)\big)    &= 0, \\
  \frac{m}{2}\, d|H|^2 + 2 \tr A_{\nabla^N_{(\,\cdot\,)} H} (\,\cdot\,)  &= 0,
 \end{array}\right.$
\end{center}
 où $H$, $B$, $A$ désignent respectivement la courbure moyenne, la seconde forme fondamentale et l'opérateur de Weingarten de $M$, la trace est prise par rapport à $g$.
\end{proposition}

\begin{proof}
 D'après l'égalité (\ref{equation_biharmonique_indicée}) et la proposition \ref{proposition_Delta_X,U,H}, la variété $M$ est biharmonique d'indice $k$ si et seulement si
 \begin{equation*}
  \Delta^N H - (m-k)\, H + \tr B\big(\,\cdot\,,A_H(\,\cdot\,)\big) + \frac{m}{2}\, d|H|^2 + 2 \tr A_{\nabla^N_{(\,\cdot\,)} H}(\,\cdot\,) = 0,
 \end{equation*}
 il suffit ensuite de décomposer sur le fibré tangent et le fibré normal de $M$.
\end{proof}

La proposition précédente nous donne deux résultats fondamentaux pour les variétés pseudo--ombilicales. Le premier est de remarquer que d'après le lemme \ref{lemme_nabla_nabla}, la deuxième équation de la proposition \ref{proposition_biharmonique_indicée} est alors équivalente à $(m-4) \,d|H|^2 = 0$. On obtient alors le théorème suivant, qui généralise aux variétés biharmoniques avec indice le théorème 5.1 de Balmus, Montaldo et Oniciuc (\cite{MR2448058}).
\begin{théorème}
 Soit $M$ une variété pseudo--ombilicale biharmonique d'indice $k$ dans $\mathbb{S}^n$ qui n'est pas de dimension $4$, alors sa courbure moyenne est de norme constante.
\end{théorème}

Notons que Balmus, Montaldo et Oniciuc ont conjecturé toujours dans \cite{MR2448058} que les variétés proprement biharmoniques dans $\mathbb{S}^n$ sont à courbure moyenne constante. Supposons de plus, que la courbure moyenne de $M$ est parallèle (c'est--à--dire  $\nabla^N H=0$) et de norme constante non nulle, alors la première équation de la proposition \ref{proposition_biharmonique_indicée} implique que $M$ est proprement biharmonique d'indice $m(1-|H|^2)$. Nous allons voir ci--dessous que sous certaines hypothèses sur $k$, cette situation décrit entièrement les variétés compactes et pseudo--ombilicales qui sont proprement biharmoniques d'indice $k$. Le résultat suivant constitue la deuxième conséquence importante de la proposition \ref{proposition_biharmonique_indicée}.
\begin{proposition}
 \label{proposition_variétés_biharmoniques_indicées_compactes}
 Supposons que $M$ est compacte,
 \begin{itemize}
  \item si $m<k$, alors $M$ est biharmonique d'indice $k$ si et seulement si $M$ est harmonique.
  \item si $m(1- \inf_{p \in M} |H_p|^2) \leq k < m$, $M$ est pseudo--ombilicale, alors $M$ est proprement biharmonique d'indice $k$ si et seulement si la courbure moyenne est parallèle, de norme constante non nulle et qu'elle vérifie $k = m(1-|H|^2)$.
 \end{itemize}
\end{proposition}

Cette proposition qui s'inspire des travaux de Caddeo, Montaldo et Oniciuc sur les applications biharmoniques (voir \cite{MR1863283} et \cite{MR2004799}) donne aussi un résultat de rigidité quand $m<k$ ; il n'existe pas de sous variétés compactes proprement biharmoniques d'indice $k$ dans $\mathbb{S}^n$. 
\begin{proof}
 Supposons que $m < k$, comme $M$ est biharmonique d'indice $k$ si et seulement si $(\Delta - m + k)\, H = 0$, il suffit de faire une intégration par parties pour montrer que $H=0$. Supposons maintenant que $k<m$ et que $M$ est pseudo--ombilicale, alors d'après la proposition \ref{proposition_biharmonique_indicée}, $M$ est biharmonique d'indice $k$ si et seulement si
 \begin{equation*}
  \Delta^N H + m |H|^2 H - (m-k)\, H = 0.
 \end{equation*}
 D'après la formule de Weitzenböck, on obtient\,:
 \begin{equation*}
  \frac{1}{2} \Delta (|H|^2) = g(\Delta^N H,H) - |\nabla^N H|^2 = - |H|^2 (m|H|^2 - m + k) - |\nabla^N H|^2
 \end{equation*}
 et on conclut encore une fois en intégrant sur $M$.
\end{proof}

\subsubsection{Etude des variations secondes}

Supposons que $M$ est compacte, on se donne $\varphi$ une application biharmonique d'indice $k$ de $\varphi$ de $M$ dans $\mathbb{S}^n$, munie d'une famille à $2$ paramètres $(\varphi_{s,t})_{s,t}$. On pose $V := \frac{\partial \varphi_{s,0}}{\partial s} \big|_{s=0}$ et $W := \frac{\partial \varphi_{0,t}}{\partial t} \big|_{t=0}$, d'après les formules sur les variations secondes des applications harmoniques et des applications biharmoniques (voir  \cite{MR1943720}), on obtient pour $\varphi$\,:
\begin{eqnarray*}
 \frac{\partial^2 \mathcal{E}^2_k(\varphi_{s,t})}{\partial s\, \partial t } \big|_{(s,t)=(0,0)} = \int_M \langle I_k(V),W \rangle \,dvol,
\end{eqnarray*}
avec
\begin{eqnarray*}
 I_k(V)
  &:=& \Delta^2 V + \Delta \big( \tr \langle V, T\varphi \rangle T\varphi - |T\varphi|^2 V \big) - 2 \langle d\Delta\varphi, T\varphi\rangle V + |\Delta\varphi|^2 V \\
  & & + 2 \tr\langle V,d\Delta\varphi\rangle T\varphi + 2 \tr\langle\Delta\varphi,dV\rangle T\varphi - \langle\Delta\varphi,V\rangle \Delta\varphi \\
  & & + \tr\langle\Delta V,T\varphi\rangle T\varphi + \tr\big\langle T\varphi, (\tr\langle V, T\varphi\rangle T\varphi) \big\rangle\, T\varphi \\
  & & - 2|T\varphi|^2 \tr\langle V,T\varphi\rangle T\varphi - 2 \langle dV,T\varphi\rangle \Delta\varphi - |T\varphi|^2 \Delta V + |T\varphi|^4 V \\
  & & + k (\Delta V - |T\varphi|^2 V + \tr\langle V,T\varphi\rangle T\varphi).
\end{eqnarray*}
Supposons que $M$ est proprement biharmonique d'indice $k$ dans $\mathbb{S}^n$, alors on obtient facilement en posant que $\varphi$ est l'inclusion de $M$ et $V:=H$\,:
\begin{equation*}
 \int_M (I_k(H),H) = - 4m^2\, \int_M |H|^4 < 0.
\end{equation*}
On vient de montrer que les variétés proprement biharmoniques d'indice $k$ sont instablement biharmoniques d'indice $k$.

\subsubsection{Un exemple pseudo--ombilical : l'hypersphère $\mathbb{S}^{n-1}(a) \subset \mathbb{S}^n$}

{\bf{Définition et propriétés}}

Soit $a\in]0,1]$, on pose $b:=\sqrt{1-a^2}$, on appelle hypersphère de rayon $a$ de $\mathbb{S}^n$, l'ensemble suivant
\begin{equation*}
 \mathbb{S}^{n-1}(a) := \{ (x,b) \in \mathbb{R}^{n+1} \,|\, x \in \mathbb{R}^n, |x| = a \} \subset \mathbb{S}^n,
\end{equation*}
et on note $H_s$ sa courbure moyenne dans $\mathbb{S}^n$. Soit $p:=(x^1,\ldots,x^n,b)$ un point de l'hypersphère de rayon $a$, alors
\begin{equation*}
 T_p \mathbb{S}(a) = \{ X=(X^1,\ldots,X^n,0) \in \mathbb{R}^{n+1} | x^1 X^1 + \cdots + x^n X^n = 0 \}.
\end{equation*}
On pose $\eta_s := - \frac{b}{a}(x^1,\ldots,x^n,-\frac{a^2}{b}) \in N_p \mathbb{S}(a)$ en remarquant que $|\eta_s|^2 = 1$. On parlera d'équateur pour l'hypersphère de rayon $1$.
\begin{proposition}
 \label{proposition_hypersphère}
 L'hypersphère $\mathbb{S}(a)$ vérifie $A_{\eta_s} = \frac{b}{a} Id$, elle est donc pseudo--ombilicale et $H_s = \frac{b}{a}\, \eta_s$. Sa courbure moyenne est donc parallèle et de norme constante. De plus, elle est harmonique si et seulement si $a=1$.
\end{proposition}

\begin{proof}
 On obtient d'une part\,:
 \begin{equation*}
  \nabla^{\mathbb{S}}_{\textstyle{X}} \eta_s
     = \nabla^{\mathbb{R}}_{\textstyle{X}} \eta_s = - \frac{b}{a} \nabla^{\mathbb{R}}_{\textstyle{(X^1,\ldots,X^n,0)}} (x^1,\ldots,x^n,-\frac{a^2}{b}) = - \frac{b}{a} X,
 \end{equation*}
 et d'autre part
 \begin{equation*}
  \nabla^{\mathbb{S}}_{\textstyle{X}} \eta_s = \nabla^N_{\textstyle{X}} \eta_s - A_{\eta_s}(X),
 \end{equation*}
  ainsi $\nabla^N \eta_s = 0$ et on détermine alors rapidement les expressions de $A_{\eta_s}$ et $H_s$ annoncées.
\end{proof}

\begin{proposition}
 Supposons que $a\neq1$, alors l'hypersphère $\mathbb{S}(a)$ est proprement biharmonique d'indice $(n-1)(1-\frac{b^2}{a^2})$.
\end{proposition}

\begin{proof}
 Application directe du corollaire \ref{corollaire_Delta_X,H_hypersphère_tore}.
\end{proof}

Soit $M$ une sous variété de $\mathbb{S}(a)$, on note $H_M$, $H_s$ et $H$ les courbures moyenne de $M \subset \mathbb{S}(a)$, $\mathbb{S}(a) \subset \mathbb{S}^n$ et $M \subset \mathbb{S}^n$, on obtient alors le lemme suivant\,:
\begin{lemme}
 \label{lemme_technique_hypersphère}
  Supposons que $a\neq1$, alors $M$ ne peut pas être harmonique à la fois dans $\mathbb{S}(a)$ et dans $\mathbb{S}^n$. En revanche, $M$ est harmonique dans l'équateur $\mathbb{S}(1)$ si et seulement si $M$ est harmonique dans $\mathbb{S}^n$. De plus, on a l'égalité suivante quelque soit $a\in]0,1]$\,:
 \begin{equation*}
   \Delta H = (\Delta^s + m\,\frac{b^2}{a^2}) \,H_M + m(|H_M|^2 + \frac{b^2}{a^2} )\, H_s.
 \end{equation*}
\end{lemme}

\begin{proof}
 Remarquons que $H = H_M + H_s$, ce qui montre la première partie du lemme. On note $\Delta$ le laplacien entre $M$ et $\mathbb{S}^n$ et $\Delta^s$ le laplacien de entre $M$ et $\mathbb{S}(a)$, alors $\Delta H_M = \Delta^s H_M + m\, |H_M|^2 H_s$ et $\Delta H_s = m\, |H_s|^2 (H_s + H_M)$ d'après le lemme \ref{lemme_nabla_nabla} et la proposition \ref{proposition_hypersphère}, ce qui suffit à montrer le lemme.
\end{proof}

Du lemme précédent, on obtient alors facilement une généralisation des théorèmes 3.5. et 3.9. de Caddeo, Montaldo et Oniciuc \cite{MR1919374} aux variétés biharmoniques avec indice par le théorème suivant\,:
\begin{théorème}
 \label{théorème_hypersphère_1}
 Supposons que $M$ soit une sous variété de $\mathbb{S}(a)$ avec $a\neq1$, alors $M$ est harmonique dans $\mathbb{S}(a)$ si et seulement si elle est proprement biharmonique d'indice $m(1-\frac{b^2}{a^2})$ dans $\mathbb{S}^n$.
\end{théorème}

\begin{proof}
 Posons $k=m(1-\frac{b^2}{a^2})$, d'après le lemme \ref{lemme_technique_hypersphère} on obtient $(\Delta - m + k) \,H = \Delta^s \,H_M + m\, |H_M|^2 \,H_s$. On conclut en remarquant que $H_s$ est non nulle et que $\Delta^s H_M$ et $H_s$ sont orthogonaux.
\end{proof}

\begin{proposition}
 Supposons que $M$ soit une sous variété de l'équateur, alors $M$ est proprement biharmonique d'indice $k$ dans l'équateur si et seulement si elle est proprement biharmonique d'indice $k$ dans $\mathbb{S}^n$.
\end{proposition}

\begin{proof}
 On a facilement que $(\Delta-m+k) H = (\Delta^s-m+k) H_M$ d'après le lemme \ref{lemme_technique_hypersphère}.
\end{proof}


Le théorème suivant montre que les hypersphères contrôlent les variétés qui vérifient les hypothèses de courbures de la proposition \ref{proposition_variétés_biharmoniques_indicées_compactes} (sans la compacité).
\begin{théorème}
 \label{théorème_hypersphère_2}
 Soit $M$ une sous variété pseudo--ombilicale de $\mathbb{S}^n$, de courbure moyenne parallèle et de norme constante non nulle, alors\,
 \begin{enumerate}
  \item $M$ est proprement biharmonique d'indice $m(1-|H|^2)$ dans $\mathbb{S}^n$,
  \item $M$ est harmonique dans $\mathbb{S}(\frac{1}{\sqrt{1+|H|^2}})$.
 \end{enumerate}
\end{théorème}

\begin{proof}
 On note $\tilde{\nabla}^N$ la connexion induite par $\nabla^{\mathbb{R}^{n+1}}$ sur le fibré normal de $TM$ dans $\mathbb{R}^{n+1}$, $\nabla^N$ la connexion induite par $\nabla^{\mathbb{S}^n}$ sur le fibré normal de $TM$ dans $\mathbb{S}^n$ et $\tilde{H}$ la courbure moyenne de $M \subset \mathbb{R}^{n+1}$. Soit $p$ un point de $M$, alors $\tilde{H}(p) = H(p) - p$. Soit $X \in TM$, alors on obtient d'une part\,:
 \begin{eqnarray*}
  \nabla^{\mathbb{R}^{n+1}}_{\textstyle{X}} \tilde{H}
   &=& \tilde{\nabla}^N_{\textstyle{X}} \tilde{H} - \tilde{A}_{\tilde{H}}(X),
 \end{eqnarray*}
 et d'autre part\,:
 \begin{equation*}
  \nabla^{\mathbb{R}^{n+1}}_{\textstyle{X}} \tilde{H}
   = \nabla^{\mathbb{S}^n}_{\textstyle{X}} H - g(X,H)\, p - \nabla^{\mathbb{R}^{n+1}}_{\textstyle{X}} p
   = \nabla^N_{\textstyle{X}} H - A_H(X) - X,
 \end{equation*}
 ce qui donne avec les hypothèses, $\tilde{\nabla}^N_{\textstyle{X}} \tilde{H} = \nabla^N_{\textstyle{X}} H = 0$ et $\tilde{A}_{\tilde{H}} = (1 + |H|^2) Id_{TM}$. Soit $\Psi$ une fonction de $M$ dans $\mathbb{R}^{n+1}$ définie par $\Psi(p) = p + \frac{1}{1+|H|^2}\, \tilde{H}(p)$, alors $\Psi$ est constante, il suffit de remarquer que quelque soit $X$ dans $TM$\,:
 \begin{equation*}
  \nabla^{\mathbb{R}^{n+1}}_{\textstyle{X}} \Psi = \nabla^{\mathbb{R}^{n+1}}_{\textstyle{X}} p + \frac{1}{1+|H|^2}\, \nabla^{\mathbb{R}^{n+1}}_{\textstyle{X}} \tilde{H} = X - \frac{1}{1+|H|^2}\, \tilde{A}_{\tilde{H}} (X) = 0.
 \end{equation*}
 De plus, $|p-\Psi|^2 = \frac{1}{1+|H|^2}$ et $|\Psi| = \frac{|H|}{\sqrt{1+|H|^2}}$, ainsi $M \subset \mathbb{S}^n(\Psi,\frac{1}{\sqrt{1+|H|^2}}) \cap \mathbb{S}^n = \mathbb{S}^{n-1}(\frac{1}{\sqrt{1+|H|^2}})$.
\end{proof}


\subsubsection{Un exemple non pseudo--ombilical : le tore généralisé de Clifford $\mathbb{T}^{n-1}(a,b) \subset \mathbb{S}^n$}

Dans tout ce paragraphe, $a$ désigne un réel dans $]0,1[$ et $(n_1,n_2)$ sont deux entiers non nuls vérifiant $n_1+n_2=n-1$, on pose $b:=\sqrt{1-a^2}$. On appelle tore généralisé de Clifford de $\mathbb{S}^n$ de rayons $a$ et $b$, l'ensemble
 \begin{eqnarray*}
  \mathbb{T}^{n_1+n_2}(a,b)
    &:=& \mathbb{S}^{n_1}(a) \times \mathbb{S}^{n_2}(b) \subset \mathbb{S}^n \\
    & =& \{ (x,y) \in \mathbb{R}^{n+1}, x \in \mathbb{R}^{n_1+1}, y \in \mathbb{R}^{n_2+1}, |x| = a, |y| = b \}
 \end{eqnarray*}
 et on note $H_t$ sa courbure moyenne dans $\mathbb{S}^n$. Soit $p:= (x,y)$ un point de $\mathbb{T}(a,b)$, alors
 \begin{equation*}
  T_p \mathbb{T}(a,b) = \{ (X,Y) \in \mathbb{R}^{n+1} \,|\, X \in \mathbb{R}^{n_1+1}, Y \in \mathbb{R}^{n_2+1}, \langle X,x \rangle = \langle Y,y \rangle = 0 \},
 \end{equation*}
  on pose $\eta_t := (-\frac{b}{a} x,\frac{a}{b} y) \in N_p\mathbb{T}(a,b)$ en remarquant que $|\eta_t| = 1$.

\begin{proposition}
 \label{proposition_tore_harmonique}
 Le tore de Clifford $\mathbb{T}(a,b)$ vérifie $A_{\eta_t} = (\frac{b}{a} Id_1,-\frac{a}{b} Id_2)$, il n'est pas pseudo--ombilical et $H_t = \frac{1}{n_1+n_2} (\frac{b}{a} \,n_1 - \frac{a}{b} \,n_2)\, \eta_t$. Sa courbure moyenne est donc parallèle et de norme constante. De plus, il est harmonique si et seulement si $b^2 \,n_1 - a^2 \,n_2 = 0$
\end{proposition}

\begin{proof}
 On obtient d'une part\,:
 \begin{equation*}
  \nabla^{\mathbb{S}}_{\textstyle{(X,Y)}} \eta_t = \nabla^{\mathbb{R}}_{\textstyle{(X,Y)}} \eta_t = (-\frac{b}{a} \,X,\frac{a}{b} \,Y),
 \end{equation*}
 et d'autre part
 \begin{equation*}
  \nabla^{\mathbb{S}}_{\textstyle{(X,Y)}} \eta_t = \nabla^N_{\textstyle{(X,Y)}} \eta_t - A_{\eta_t}((X,Y)),
 \end{equation*}
 ainsi $\nabla^N \eta_t = 0$ et on détermine rapidement les expressions de $A_{\eta_t}$ et $H_t$ annoncées.
\end{proof}

\begin{proposition}
 \label{proposition_tore_biharmonique}
 Supposons que $\mathbb{T}(a,b)$ est non--harmonique, alors il est biharmonique d'indice $(1 - \frac{b^2}{a^2})\, n_1 + (1 - \frac{a^2}{b^2})\, n_2$.
\end{proposition}

\begin{proof}
 Application directe du corollaire \ref{corollaire_Delta_X,H_hypersphère_tore}.
\end{proof}

\begin{proposition}
 \label{proposition_borne_tore}
 Si $k > n_1 + n_2 - 2\sqrt{n_1 n_2}$, il n'existe pas de tore de Clifford qui soit proprement biharmonique d'indice $k$.

 Si $k = n_1 + n_2 - 2\sqrt{n_1 n_2}$, alors il existe un unique tore de Clifford qui est proprement biharmonique d'indice $k$.

 Si $k < n_1 + n_2 - 2\sqrt{n_1 n_2}$, alors il existe exactement deux tores de Clifford qui sont proprement biharmoniques d'indice $k$ (si $n_1=n_2$ on a $\mathbb{T}(a,b)$ et $\mathbb{T}(b,a)$).
\end{proposition}

\begin{proof}
 Soit $\mathbb{T}(a,b)$ un tore de Clifford non harmonique, d'après la proposition \ref{proposition_tore_biharmonique}, il est alors proprement biharmonique d'indice $k$ si et seulement si $(\frac{b^2}{a^2} - 1) n_1 + (\frac{a^2}{b^2} - 1) n_2 + k = 0$ avec $k\neq0$. En posant $z := \frac{a^2}{b^2}$, il s'agit alors de résoudre l'équation du second degré $n_2z^2 -(n_1+n_2-k) z + n_1 = 0$. On note $d:=(n_1 + n_2 - k)^2 - 4n_1n_2$ son discriminant, alors $d \geq 0$ si et seulement si $k \leq n_1 + n_2 - 2\sqrt{n_1 n_2}$ ou $k \geq n_1 + n_2 + 2\sqrt{n_1 n_2}$ et dans ces cas là, on obtient les deux racines suivantes\,:
 \begin{equation*}
  z_{\pm} = \frac{n_1+n_2-k \pm \sqrt{(n_1+n_2-k)^2 - 4n_1n_2}}{2n_2}.
 \end{equation*}
 Si $k \geq n_1 + n_2 + 2\sqrt{n_1 n_2}$, alors $z_{\pm} < 0$, ce qui est impossible; il n'existe donc pas de tore de Clifford proprement biharmonique d'indice strictement plus grand que $n_1 + n_2 - 2\sqrt{n_1 n_2}$. Supposons maintenant que $k \leq n_1 + n_2 - 2\sqrt{n_1 n_2}$, alors $z_{\pm} > 0$ et on obtient deux tores de Clifford proprement biharmoniques d'indice $k$ qui sont $\mathbb{T}\big(\sqrt{\frac{1}{1+z_+}} , \sqrt{\frac{z_+}{1+z_+}}\, \big)$ et $\mathbb{T}\big(\sqrt{\frac{1}{1+z_-}} , \sqrt{\frac{z_-}{1+z_-}}\, \big)$. On conclut en remarquant que $\frac{z_+}{1+z_+} = \frac{n_1}{n_1+n_2z_-}$ et que si $k = n_1 + n_2 - 2\sqrt{n_1 n_2}$, alors $z_+ = z_-$.
\end{proof}

Considérons maintenant le cas des sous variétés d'un tore de Clifford, on se donne $M_1$ une sous variété de $\mathbb{S}^{n_1}(a)$ de dimension $m_1$ et $M_2$ une sous variété de $\mathbb{S}^{n_2}(b)$ de dimension $m_2$ avec $0<m_1<n_1$ et $0<m_2<n_2$. On note respectivement $H$, $H_1$ et $H_2$ les courbures moyenne de $M_1\times M_2 \subset \mathbb{S}^n$, $M_1 \subset \mathbb{S}^{n_1}(a)$ et $M_2 \subset \mathbb{S}^{n_2}(b)$.
\begin{lemme}
 \label{lemme_technique_tore}
 Avec les notations précédentes, on a l'égalité suivante\,:
 \begin{eqnarray*}
  m\, \Delta H
    &=& m_1 (\Delta_1^s + \frac{b^2}{a^2} m_1 - m_2)\, H_1 + m_2 (\Delta_2^s + \frac{a^2}{b^2} m_2 - m_1)\, H_2 \\
    & & + \big( (\frac{b}{a}\, m_1 - \frac{a}{b}\, m_2) (\frac{b^2}{a^2}\, m_1 + \frac{a^2}{b^2}\, m_2) + \frac{b}{a}\, m_1 |H_1|^2 - \frac{a}{b}\, m_2 |H_2|^2\big) \eta_t.
 \end{eqnarray*}
\end{lemme}

\begin{proof}
 On a pour la courbure moyenne $H$ \,:
 \begin{equation}
  \label{egalité_H_H_M_tore}
  m\, H = (\frac{b}{a}\, m_1 - \frac{a}{b}\, m_2)\, \eta_t + m_1 H_1 + m_2 H_2.
 \end{equation}
 Soit $\Delta$ le laplacien de $M$ dans $\mathbb{S}^n$, $\Delta_1^s$ le laplacien de $M_1$ dans $\mathbb{S}^{n_1}(a)$ et $\Delta_2^s$ le laplacien de $M_2$ dans $\mathbb{S}^{n_2}(b)$, le lemme \ref{lemme_nabla_nabla} et la proposition \ref{proposition_tore_harmonique} nous donne alors les trois égalités suivantes, qui permettent de conclure la démonstration du lemme\,:
 \begin{eqnarray*}
  \Delta H_1      &=& \Delta_1^s H_1 + \frac{b}{a}\, m_1 |H_1|^2 \eta_t, \\
  \Delta H_2      &=& \Delta_2^s H_2 - \frac{a}{b}\, m_2 |H_2|^2 \eta_t, \\
  \Delta \eta_t   &=& (\frac{b^2}{a^2}\, m_1 + \frac{a^2}{b^2}\, m_2 )\, \eta_t + \frac{b}{a}\, m_1 H_1 - \frac{a}{b}\, m_2 H_2.
 \end{eqnarray*}
\end{proof}

Il découle du lemme précédent, une généralisation des théorèmes 3.11. et 3.13. de Caddeo, Montaldo et Oniciuc \cite{MR1919374} aux variétés biharmoniques avec indice, par le théorème suivant\,:


\begin{théorème}
 \label{théorème_tore_1}
 La variété $M_1 \times M_2$ est proprement biharmonique d'indice $(1-\frac{b^2}{a^2})\, m_1 + (1-\frac{a^2}{b^2})\, m_2$ dans $\mathbb{S}^n$ si et seulement si les trois hypothèses suivantes sont vérifiées\,:
 \begin{enumerate}
  \item $M_1$ est harmonique dans $\mathbb{S}^{n_1}(a)$,
  \item $M_2$ est harmonique dans $\mathbb{S}^{n_2}(b)$,
  \item $b^2 m_1 - a^2 m_2 \neq 0$.
 \end{enumerate}
\end{théorème}

\begin{proof}
  Posons $k = (1-\frac{b^2}{a^2})\, m_1 + (1-\frac{a^2}{b^2})\, m_2$, le lemme \ref{lemme_technique_tore} nous donne l'égalité suivante\,:
 \begin{eqnarray*}
  m\, (\Delta - m + k)\, H
    &=& m_1\, \big(\Delta_1^s - (1 + \frac{a^2}{b^2})\, m_2\big) H_1 + m_2\, \big(\Delta_2^s - (1 + \frac{b^2}{a^2})\, m_1\big) H_2 \\
    & & + (\frac{b}{a}\,m_1 |H_1|^2 -\frac{a}{b}\,m_2 |H_2|^2)\, \eta_t.
 \end{eqnarray*}
 Il suffit ensuite de remarquer que $\Delta_1^s H_1$, $H_1$, $\Delta_2^s H_2$, $H_2$ et $\eta_t$ sont tous orthogonaux deux à deux et on conclut avec l'égalité (\ref{egalité_H_H_M_tore}).
\end{proof}

\subsection{Les sous variétés proprement C--harmoniques de $\mathbb{S}^n$}

\subsubsection{Les applications C--harmoniques}

Sur une surface de Riemann, l'énergie d'une application à valeurs dans une variété riemannienne est une fonctionnelle invariante conforme, ses points critiques sont les applications harmoniques. Il existe un analogue en dimension supérieure, résumé par le théorème suivant\,:

\begin{théorème_vide}[Bérard, \cite{MR2449641}]
 Soit $(M^n,g)$ et $(N,h)$ deux variétés riemanniennes, on suppose que $n$ est pair, alors il existe une fonctionnelle sur les applications de classe $C^{\infty}$ de $(M,g)$ dans $(N,h)$ qui est invariante conforme par rapport à $g$. De plus, l'équation de ses points critiques est une EDP elliptique non--linéaire d'ordre $n$, qui est invariante conforme elle aussi par rapport à $g$.

 Les applications conforme--harmoniques (qu'on abrège en C--harmoniques) sont définis comme les points critiques de cette fonctionnelle.
\end{théorème_vide}
La construction est explicite, nous nous contenterons de donner la définition en petites dimensions (voir \cite{MR2722777} pour plus détails).

\subsubsection{Les variétés proprement C--harmoniques de dimension $4$}

Rappelons la définition d'une application C--harmonique quand la variété de départ est de dimension $4$ (voir théorème 5.1.1 dans \cite{MR2722777})\,:
\begin{définition_vide}
 Soient $(M^4,g)$ et $(N,h)$ deux variétés riemanniennes, une application $\varphi$ de $M$ dans $N$ est dite C--harmonique de $(M,[g])$ dans $(N,h)$ si
 \begin{equation}
  \label{equation_C-harmonique_dim4}
  \delta d\delta T\varphi + \delta (\frac{2}{3}\scal - 2\ric) T\varphi - \Se(\delta T\varphi) = 0,
 \end{equation}
 où $\ric$, $\scal$ et $dvol$ se rapportent à $g$, $\R^h$ est le tenseur de courbure de $(N,h)$ et $(e_1,\ldots,e_4)$ est une base orthonormée de $TM$ par rapport à $g$.
\end{définition_vide}

Dans le cas des sous variétés de $\mathbb{S}^n$, la C--harmonicité donne la proposition suivante\,:
\begin{proposition}
 \label{proposition_compacte_C--harmonique_4}
 Soit $M^4$ une sous variété compacte de $\mathbb{S}^n$,
 \begin{itemize}
  \item si $n\geq6$, alors $M$ est une C--harmonique si et seulement si elle est harmonique.
  \item si $n=5$, $M$ est pseudo--ombilicale et $|H|^2\geq 1/6$, alors elle est proprement C--harmonique si et seulement si $|H|^2=1/6$ et $\nabla^N H=0$.
 \end{itemize}
\end{proposition}

C'est--à--dire, qu'il n'existe pas de sous variétés compactes de dimension $4$ de $\mathbb{S}^n$ proprement C--harmoniques si $n\geq6$. En revanche les sous variétés pseudo--ombilicales de $\mathbb{S}^5$ qui vérifient $|H| \geq \frac{1}{6}$ sont proprement C--harmoniques si et seulement si $|H|=\frac{1}{6}$.
\begin{proof}
 Il suffit d'appliquer la proposition \ref{proposition_variétés_biharmoniques_indicées_compactes} avec $m=4$ et $k=\frac{\scal}{6}=\frac{n(n-1)}{6}$.
\end{proof}

On obtient alors pour les hypersphères, notre exemple type de sous variétés pseudo--ombilicales de $\mathbb{S}^n$\,:
\begin{corollaire}
 La seule hypersphère de dimension $4$ qui est proprement C--harmonique dans $\mathbb{S}^5$ est $\mathbb{S}^4(\sqrt{\frac{6}{7}})$.
\end{corollaire}

\begin{proof}
 D'après la proposition \ref{proposition_hypersphère}, $k=\frac{10}{3}$ et donc $a = \frac{1}{\sqrt{2-\frac{k}{4}}} = \sqrt{\frac{6}{7}}$.
\end{proof}

En ce qui concerne les variétés qui ne sont pas pseudo--ombilicales, on a le résultat de rigidité suivant pour les tores de Clifford\,:
\begin{proposition}
 \label{proposition_tore_dim4}
 Il n'existe pas de tore de Clifford de dimension $4$ qui soit proprement C--harmonique dans $\mathbb{S}^5$.
\end{proposition}

\begin{proof}
 Il suffit de remarquer que $k=\frac{10}{3}>(1-\sqrt{3})^2$ et d'appliquer la proposition \ref{proposition_tore_biharmonique}.
\end{proof}

\subsubsection{Les variétés proprement C--harmoniques de dimension $6$}

Rappelons la définition d'une application C--harmonique quand la variété de départ est de dimension $6$ (voir théorème 5.2.1 dans \cite{MR2722777})\,:
\begin{définition_vide}
 Soient $(M^6,g)$ une variété d'Einstein et $(N,h)$ une variété riemannienne symétrique, une application $\varphi$ qui est $C^{\infty}$ de $M$ dans $N$ est dite C--harmonique de $(M,[g])$ dans $(N,h)$ si elle vérifie l'équation suivante\,:
 \begin{equation} 
  \label{equation_C-harmonique_dim6}
  (\delta d - \Se + \frac{2\scal}{15})(\delta d - \Se + \frac{\scal}{5})\,\delta T\varphi - 2 \tr \R^h_{\textstyle{\delta T\varphi,T\varphi}} (\nabla^h_{\!\textstyle{T\varphi}} \delta T\varphi) = 0
 \end{equation}
 où $\nabla^h$ est la connexion de $h$ sur $\varphi^*TN$, $\scal$ est la courbure scalaire de $(M,g)$, $R^h$ est le tenseur de courbure de $(N,h)$ et la trace est prise par rapport à $g$.
\end{définition_vide}

Dans le cas des sous variétés de dimension $6$ de $\mathbb{S}^n$, la C--harmonicité donne la proposition suivante\,:
\begin{proposition}
 Soit $M^6$ une sous variété pseudo--ombilicale de $\mathbb{S}^n$, alors $M$ est C-harmonique si et seulement si\,:
 \begin{eqnarray}
  \label{equation_sous_variété_C-harmonique_dim6}
  \lefteqn{\Delta^2 H + \frac{1}{3}\, (n^2-n-36) \Delta H - 72\, |H|^2\, H} \notag\\
   & & + \frac{2}{75}\, (n^2-n-45)(n^2-n-30)\, H - 6\, d|H|^2 = 0.
 \end{eqnarray}
\end{proposition}

\begin{proof}
  D'après (\ref{equation_C-harmonique_dim6}), $M$ est C--harmonique si et seulement si\,:
 \begin{equation*}
  \big(\Delta + \frac{2}{15}\, (n^2-n-45)\big) \big(\Delta + \frac{1}{5}\, (n^2-n-30)\big)\,H + 12 \tr \R_{\textstyle{H,\,\cdot\,}} (\nabla_{\textstyle{(\,\cdot\,)}} H) = 0,
 \end{equation*}
 il suffit alors de remarquer que le dernier terme est égale à $- 72\, |H|^2\, H - 6\, d|H|^2$.
\end{proof}

\begin{proposition}
 \label{proposition_compacte_C--harmonique_6}
 Soit $M^6$ une sous variété compacte pseudo--ombilicale de $\mathbb{S}^n$,
 \begin{itemize}
  \item si $n\geq 8$, alors $M$ est C--harmonique si et seulement si elle est harmonique.
  \item si $n=7$ et $|H|^2 \geq \frac{1}{30} (25 + \sqrt{649})$, alors $M$ est proprement C--harmonique si et seulement si $|H|^2 = \frac{5}{6} + \frac{\sqrt{649}}{30}$ et $\nabla^N H=0$.
 \end{itemize}
\end{proposition}

C'est--à--dire, que pour $n\geq8$, il n'existe pas de sous variétés compactes de dimension $6$ proprement C--harmonique dans $\mathbb{S}^n$ ayant une courbure moyenne \og petite \fg.
\begin{proof}
 Supposons que $M$ est C--harmonique, alors en prenant le produit scalaire du terme de gauche de l'égalité (\ref{equation_sous_variété_C-harmonique_dim6}) contre $H$ qu'on intègre sur $M$, on obtient
 \begin{eqnarray}
  \label{egalité_C-harmonique_dim6_compacte}
  0 &=& \int_M \big(|\Delta H|^2  + \frac{1}{3}\, (n^2-n-36) (|\nabla ^N H|^2 + |A_H|^2)\big) \,dvol \notag\\
    & & - \int_M \big(72\, |H|^4 - \frac{2}{75}\, (n^2-n-45)(n^2-n-30)\, |H|^2\big) \,dvol.
 \end{eqnarray}
 Comme $M$ est pseudo--ombilicale, la proposition (\ref{proposition_Delta_X,U,H}) et la formule de Weitzenböck nous donne\,:
 \begin{equation*}
  \int_M |\Delta H|^2 \,dvol = \int_M \big(|\Delta^N H|^2 + 7\, |d|H|^2|^2 + 12\, |H|^2 |\nabla^N H|^2 + 36\, |H|^6 \,dvol,
 \end{equation*}
 ce qui donne pour le terme de droite de (\ref{egalité_C-harmonique_dim6_compacte})\,:
 \begin{equation*}
   \int_M \Big(|\Delta^N H|^2 + 7\, |d|H|^2|^2 + \big(12\, |H|^2 + \frac{1}{3}\, (n^2-n-36)\big) |\nabla^N H|^2\Big) \,dvol
 \end{equation*}
 \begin{equation*}
   + \int_M |H|^2 \big(36\, |H|^4 + 2(n^2-n-72)\, |H|^2 + \frac{2}{75}\, (n^2-n-45)(n^2-n-30)\big) \,dvol.
 \end{equation*}
 Remarquons que la première intégrale est positive ou nulle, il suffit d'étudier le signe de la deuxième. Si $n\geq9$, tous ses coefficients sont strictement positifs, ainsi $H=0$. Pour $n=8$, la deuxième intérale est égale à
 \begin{equation*}
  \int_M |H|^2 \big( 36\, |H|^4 - 32\, |H|^2 + \frac{572}{75} \big) \,dvol
 \end{equation*}
 et comme $36\, |H|^4 - 32\, |H|^2 + \frac{572}{75} > 0$, alors $H=0$. Quand $n=7$, on obtient alors pour la deuxième intégrale\,:
 \begin{equation*}
  \int_M |H|^2 \big( 36\, |H|^4 - 60\, |H|^2 - \frac{24}{25} \big) \,dvol,
 \end{equation*}
 il suffit de remarquer que la plus grande racine du trinôme $36\, x^2 - 60\, x - \frac{24}{25}$ est $\frac{1}{30} (25 + \sqrt{649})$ pour conclure.
\end{proof}

On obtient alors pour les hypersphères, notre exemple type de sous variétés pseudo--ombilicales de $\mathbb{S}^n$\,:
\begin{corollaire}
 La seule hypersphère de dimension $6$ qui soit proprement C--harmonique dans $\mathbb{S}^7$ est $\mathbb{S}^6(\frac{1}{6} \sqrt{25-5\sqrt{\frac{59}{11}}})$.
\end{corollaire}

\begin{proof}
 D'après la proposition \ref{proposition_hypersphère} et le corollaire \ref{corollaire_Delta_X,H_hypersphère_tore}, l'hypersphère $\mathbb{S}(a)$ est proprement C--harmonique si et seulement si
 \begin{equation*}
  \big(|H|^2 - \frac{1}{30}(25 + \sqrt{649}) \big) \big(|H|^2 - \frac{1}{30}(25 - \sqrt{649})\big) = 0,
 \end{equation*}
 ainsi $|H|^2 = \frac{1}{30}(25 + \sqrt{649})$ ce qui donne l'hypersphère $\mathbb{S}(\frac{1}{6} \sqrt{25-5\sqrt{\frac{59}{11}}})$.
\end{proof}

En ce qui concerne les variétés qui ne sont pas pseudo--ombilicales, on a le résultat de rigidité suivant pour les tores de Clifford\,:
\begin{proposition}
 \label{proposition_tore_dim6}
 Il n'existe pas de tore de Clifford de dimension $6$ qui soit proprement C--harmonique dans $\mathbb{S}^7$.
\end{proposition}

\begin{proof}
 D'après la proposition \ref{proposition_tore_biharmonique} et le corollaire \ref{corollaire_Delta_X,H_hypersphère_tore}, le tore de Clifford $\mathbb{T}(a,b)$ est proprement C--harmonique si et seulement si
 \begin{eqnarray*} 
  0 &=& n_1^2z^2 - 2n_1(n_1-1)z + 6n_1n_2 - \frac{24}{25} - \frac{2n_2(n_2-1)}{z} + \frac{n_2^2}{z^2} \\
    &=& (n_1z - n_1 + 1)^2 + 8 (n_1n_2 - \frac{337}{100}) +(\frac{n_2}{z} - n_2 + 1)^2,
 \end{eqnarray*}
 où l'on a posé $z:=\frac{a^2}{b^2}$, il suffit de remarquer que $n_1n_2\geq5$.
\end{proof}

\subsection{Appendice}

\begin{lemme}
 \label{lemme_nabla_nabla}
 Soit $M^m$ une sous variété de $\mathbb{S}^n$, alors on a les égalités suivantes\,:
 \begin{eqnarray*}
    \nabla_{\textstyle{V}} \nabla_{\textstyle{V}} U
     &=& \nabla^T_{\textstyle{V}} \nabla^T_{\textstyle{V}} U - A_{B(V,U)}(V) + 2\, B(V,\nabla^T_{\textstyle{V}} U) + \nabla^N_{\textstyle{U}} B(V,V), \\
   \nabla_{\textstyle{V}} \nabla_{\textstyle{V}} X
     &=& \nabla^N_{\textstyle{V}} \nabla^N_{\textstyle{V}} X - B\big(V,A_X(V)\big) - g\big(\nabla^N B(V,V),X\big) - 2\, A_{\nabla^N_{V} X}(V),
 \end{eqnarray*}
 où $(U,V) \in TM$ et $X \in NM$. Si $M$ est pseudo--ombilicale, alors on obtient\,:
 \begin{equation*}
   \tr A_{\nabla^N_{(\,\cdot\,)} H}(\,\cdot\,) = - \frac{m-2}{2}\, d|H|^2.
 \end{equation*}
\end{lemme}

\begin{proof} 
 Soit $(U,V) \in TM$, par un calcul direct, on montre que
 \begin{equation*}
  \nabla_{\textstyle{V}} \nabla_{\textstyle{V}} U = \nabla^T_{\textstyle{V}} \nabla^T_{\textstyle{V}} U - A_{B(V,U)}(V) + B(V,\nabla^T_{\textstyle{V}} U) + \nabla^N_{\textstyle{V}} B(V,U),
 \end{equation*}
 alors en intervertissant les dérivées on obtient\,:
 \begin{eqnarray}
 \label{egalite_nabla_B}
  \nabla^N_{\textstyle{V}} B(V,U) 
    &=& p_N\big( \nabla_{\textstyle{U}} \nabla^T_{\textstyle{V}} V + \nabla_{\textstyle{U}} B(V,V)  \big) + B(\nabla^T_{\textstyle{V}} U,V) \notag\\
    &=& \nabla^N_{\textstyle{U}} B(V,V) + B(\nabla^T_{\textstyle{V}} U,V),
 \end{eqnarray}
 ce qui donne la première égalité. Soit $X \in NM$, alors on obtient directement\,:
 \begin{equation*}
  \nabla_{\textstyle{V}} \nabla_{\textstyle{V}} X = \nabla^N_{\textstyle{V}} \nabla^N_{\textstyle{V}} X - B\big(V,A_X(V)\big) - \nabla^T_{\textstyle{V}} A_X(V) - A_{\nabla^N_{V} X}(V).
 \end{equation*}
 et d'autre part, avec l'égalité (\ref{egalite_nabla_B}),
 \begin{eqnarray}
  \label{egalite_nabla_A}
  \nabla^T_{\textstyle{V}} A_{\textstyle{X}} (V)
    &=& g\big(\nabla^N_{\textstyle{V}} B(V,\,\cdot\,),X\big) + A_{\nabla^N_{V} X}(V) - g\big(B(V,\nabla^T_{\textstyle{V}} \,\cdot\,),X \big) \notag\\
    &=& g\big(\nabla^N B(V,V),X\big) + A_{\nabla^N_{V} X}(V),
 \end{eqnarray}
 ce qui donne la deuxième égalité. Pour la troisième égalité, si $M$ est pseudo--ombilicale, alors $\tr \nabla^T_{(\,\cdot\,)} A_H (\,\cdot\,) = d|H|^2$, on pose $X=H$ dans (\ref{egalite_nabla_A}) et on prend la trace par rapport à $g$.
\end{proof}

\begin{proposition}
 \label{proposition_Delta_X,U,H}
 Soit $M^m$ une sous variété de $\mathbb{S}^n$, alors on a les égalités suivantes\,:
 \begin{eqnarray*} 
  \Delta X &=& \Delta^N X + \tr B\big(\,\cdot\,,A_X(\,\cdot\,)\big) + m\, g(\nabla_{\textstyle{.}}H,X) + 2\, \tr A_{\nabla^N_{(\,\cdot\,)}X}(\,\cdot\,), \\
  \Delta U &=& \Delta^T U + \tr A_{B(\,\cdot\,,U)}(\,\cdot\,) - m\nabla^N_{\textstyle{U}} H - 2\, \tr B(\,\cdot\,,\nabla^T_{(\,\cdot\,)} U) \\
  \Delta H &=& \Delta^N H + \tr B\big(\,\cdot\,,A_H(\,\cdot\,)\big) + \frac{m}{2}\, d|H|^2 + 2\, \tr A_{\nabla^N_{(\,\cdot\,)}H}(\,\cdot\,),
 \end{eqnarray*}
 où $X\in NM$, $U\in TM$ et la trace est prise par rapport à $g_{\mathbb{S}}$. Si $M$ est une sous variété pseudo--ombilicale, alors, on obtient pour le laplacien de $H$\,:
 \begin{equation*}
  \Delta H = \Delta^N H + m|H|^2 H - \frac{m-4}{2}\, d|H|^2.
 \end{equation*}
\end{proposition}

\begin{proof}
 Application directe du lemme \ref{lemme_nabla_nabla}.
\end{proof}

Dans le cas de l'hypersphère et du tore de Clifford, on obtient\,:
\begin{corollaire}
 \label{corollaire_Delta_X,H_hypersphère_tore}
 On obtient $\Delta H_s = (n-1)\, \frac{b^2}{a^2}\, H_s$ dans le cas de l'hypersphère $\mathbb{S}(a)$ et $\Delta H_t = \big(\frac{b^2}{a^2}\, n_1 + \frac{a^2}{b^2}\, n_2 \big)\, H_t$ dans le cas du tore de Clifford $\mathbb{T}(a,b)$.
\end{corollaire}

\end{document}